\documentclass[a4paper,twoside,12pt]{article}
\usepackage[english]{babel}

\usepackage{amsmath,amsfonts,amssymb,amsthm}
\newtheorem{teo}{Theorem}
\newtheorem{lem}{Lemma}

 \title{{\bf   Cauchy Problem for an abstract Evolution Equation of  fractional order      }}

\author{Maksim \,V.~Kukushkin   \\ \\
 \small  \textit{Moscow State University of Civil Engineering, 129337,  Moscow, Russia}\\
 \textit{\small\textit{kukushkinmv@rambler.ru}} }

\date{}
\begin{document}

\maketitle

\begin{abstract}
In this paper, we define an operator function as a series of operators corresponding to the Taylor series representing the function of the complex variable. In previous papers, we considered the case when a function has a decomposition in the Laurent series with the infinite principal part and finite regular part. Our  central challenge  is to improve this result having  considered as a regular part an entire function satisfying the special condition of the growth regularity. As an application we consider an opportunity to broaden the conditions imposed upon the second term not containing the time variable of the evolution equation in the abstract Hilbert space.

\end{abstract}
\begin{small}\textbf{Keywords:}
 Evolution equations;  Operator function;  Fractional differential equations;  Abel-Lidskii basis property;   Schatten-von Neumann  class.   
\end{small}

\section{Introduction}

The urbanization of the sea coast and the active use of shelf resources have led to an increase in accidents of man-made origin and their negative impact on the environment. In these conditions, it becomes critically important to control natural catastrophic phenomena and assess their consequences in order to minimize possible human and material losses. This was the reason for the development in recent years of acoustic measurement methods and long-term monitoring of the parameters of the aquatic environment in shallow waters. Hydroacoustic complexes operating on the basis of the proposed methods provide remote measurement of the parameters of the aquatic environment in bays, straits and inland reservoirs, allowing early detection of sources of natural and man-made threats.
Ferroelectrics are a promising class of polar dielectrics and  the study of their  nonequilibrium dynamics, phase transitions and domain kinetics is of key importance in acoustics. In  the paper \cite{L. Mor}, the description of the process of switching the polarization of ferroelectrics is implemented by modeling a fractal system. Since the polarization switching process is the result of the formation of self-similar structures, the domain configurations of many ferroelectrics are characterized by a self-similar structure, and electrical responses are characterized by fractal patterns. The manifestation of fractal properties is due to the complex mechanisms of domain boundary movement, the anisotropy properties of real crystals, the stochastic nature of the nucleation process, and the presence of memory effects. The field of application of the results of fractal system modeling is focused on the description of the process of switching the polarization of ferroelectrics. The main mathematical object of research is the Cauchy problem for the evolution  equation with a fractional Riemann-Liouville   derivative in the first term. The proposed solution methods are numerical, based on a finite difference method.

   At the same time the method invented in this paper allows us to solve such problems analytically what is undoubtedly a great  advantage.
   For instance, we have an opportunity to prove      the existence and uniqueness theorems  for evolution  equations   with the second term represented by an operator function of a differential operator  with a fractional derivative in  final terms. Herewith,  the well-known operators such as  the Riemann-Liouville  fractional differential operator, the Kipriyanov operator,  the Riesz potential, the difference operator  (more  detailed \cite{firstab_lit:1kipriyanov1960},\cite{kukushkin2021a},\cite{firstab_lit:samko1987}), the artificially constructed normal operator  are involved \cite{firstab_lit:2kukushkin2022}. Here, we represent  a list of papers  dealing  with the problems  which can be investigated  by the obtained in this paper abstract method \cite{firstab_lit:Andronova2017,firstab_lit:Mamchuev2017a,firstab_lit:Mamchuev2017,firstab_lit:Pskhu}.
The physical significance of the problem  is based upon the broad field of applications, here we referee the most  valuable example considered above \cite{L. Mor} in order to show plainly  the significance. The main idea of the results connected with the basis property in the Abell-Lidskii sense  \cite{firstab_lit:1kukushkin2021} allow us to solve  many   problems \cite{firstab_lit:2kukushkin2022} in the theory of evolution  equations and in this way obtain remarkable  applications. The central aim  of this paper is devoted to  an approach  allowing us to principally broaden conditions imposed upon the second term   of the evolution equation in the abstract Hilbert space. In this way we can obtain abstract results covering many applied problems to say nothing on the far-reaching   generalizations.
We plan to implement the idea  having involved a notion of  an operator function and our final goal is  an existence and uniqueness theorem for an abstract evolution equation with an operator function in the second term not containing the time variable, where the derivative in the first term is supposed to be of fractional order. The  peculiar result that is worth highlighted  is the obtained analytic formula for the solution.   We should remind that involving a notion of the operator function, we broaden a great deal  an operator  class corresponding to the second term.   Thus, we can state   that the main issue of the paper is closely connected  with root vectors expansion of the non-selfadjoint unbounded  operator.
We should note that    the question regarding decompositions of the non-selfadjoint operator on the series of eigenvectors  (root vectors)   is rather complicated and  deserves to be considered itself. For this purpose, we need to involve some generalized notions of the series convergence, we mean   Abel-Lidskii  sense  of the series convergence  \cite{firstab_lit:1Lidskii},\cite{firstab_lit:2Agranovich1994},\cite{firstab_lit:1Gohberg1965}. Here, we should referee remarkable papers and monographs in the framework of the theory   that allows us to obtain such exotic results \cite{firstab_lit:1Lidskii},\cite{firstab_lit:2Agranovich1994}, \cite{firstab_lit(arXiv non-self)kukushkin2018}, \cite{kukushkin2019}, \cite{kukushkin2021a}, \cite{firstab_lit:1kukushkin2021}, \cite{firstab_lit:2kukushkin2022}, \cite{firstab_lit(axi2022)}, \cite{firstab_lit:Markus Matsaev}.
To be exeat, a peculiar interest appears in the case when   a senior term of an operator is not selfadjoint  for there is a number of papers \cite{firstab_lit:1Katsnelson},\cite{firstab_lit:1Krein},\cite{firstab_lit:Markus Matsaev},\cite{firstab_lit:2Markus},\cite{firstab_lit:Motovilov},
\cite{firstab_lit:Shkalikov A.} devoted to the perturbed selfadjoint operators.
It is remarkable that in this regard, we can justify  the methods used in the paper
   \cite{firstab_lit(arXiv non-self)kukushkin2018} for they have a natural mathematical origin that appears brightly  while we are considering abstract constructions expressed  in terms of the semigroup theory   \cite{kukushkin2021a}.

\section{Preliminaries}

\noindent{\bf   The growth scale}\\

To characterize the growth of an entire function $f(z),$ we introduce the functions
$$
M_{f}(r)=\max\limits_{|z|=r}|f(z)|,\,m_{f}(r)=\min\limits_{|z|=r}|f(z)|.
$$
An entire function $f(z)$ is said to be a function of finite order if there exists a positive constant $k$ such that the inequality
 $$
 M_{f}(r)<e^{r^{k}}
 $$
 is valid for all sufficiently large values of $r.$ The greatest lower bound of such numbers $k$ is called the {\it order} of the entire function $f(z).$\\

It follows from the definition that if $\varrho$ is the order of the entire function $f(z),$ and if $\varepsilon$ is an arbitrary positive number, then
\begin{equation}\label{1j}
e^{r^{\varrho-\varepsilon}}<M_{f}(r)<e^{r^{\varrho+\varepsilon}},
\end{equation}
where the inequality on the right-hand side is satisfied for all sufficiently large values of $r,$ and the inequality on the left-hand side holds for some sequence $\{r_{n}\}$ of values of $r,$ tending to infinity. It is easy to verify that the previous condition is equivalent to the equation
$$
\varrho=\overline{\lim\limits_{r\rightarrow \infty}}\,\frac{\ln\ln M_{f}(r)}{\ln r},
$$
which is taken as the definition of the order of the function. Further, an inequality that holds for all sufficiently large values of $r$ will be called an {\it asymptotic inequality.}

For functions of the given order a more precise characterization of the growth is given by the type of the function. By the {\it type} $\sigma$ of the entire  function $f(z)$ of the order $\varrho$ we mean the greatest lower bound of positive numbers $A$ for which the following relation  holds asymptotically
 $$
 M_{f}(r)<e^{Ar^{\varrho}}.
 $$
Analogously to the definition of the order it is easy to verify that the type $\sigma$ of a function $f(z)$ of order $\varrho$ is given by the relation
$$
\sigma=\overline{\lim\limits_{r\rightarrow\infty}}\frac{\ln M_{f}(r)}{r^{\varrho}}.
$$

\noindent{\bf   Convergence exponent}\\

Here,    following  the monograph \cite{firstab_lit:Eb. Levin}, we introduce some notions and facts of the entire function theory. In this subsection, we   use the following notations
$$
G(z,p):=(1-z)e^{z+\frac{z^{2}}{2}+...+\frac{z^{p}}{p}},\,p\in \mathbb{N},\,G(z,0):=(1-z).
$$
Consider   an entire function  that has zeros satisfying the following relation for some   $\lambda>0$
\begin{equation}\label{1}
 \sum\limits_{n=1}^{\infty}\frac{1}{|a_{n}|^{\lambda}}<\infty.
\end{equation}
In this case, we denote by $p$ the smallest integer number for which the following condition holds
\begin{equation}\label{2}
\sum\limits_{n=1}^{\infty}\frac{1}{|a_{n}|^{p+1}}<\infty .
\end{equation}
 It is clear that $0\leq p<\lambda.$ It is proved (see \cite{firstab_lit:Eb. Levin}) that under the assumption \eqref{1}  the   infinite product
 \begin{equation}\label{3}
 \prod\limits_{n=1}^{\infty} G\left(\frac{z}{a_{n}},p\right)
\end{equation}
 is uniformly convergent, we will call it a canonical product and call $p$ the genus of the canonical product.
By the   {\it convergence exponent} $\rho$ of the sequence
$
\{a_{n}\}_{1}^{\infty}\subset \mathbb{C},\,a_{n}\neq 0,\,a_{n}\rightarrow \infty
$
 we mean the greatest lower bound for such numbers $\lambda$ that the   series \eqref{1} converges.
 Note that if $\lambda$ equals to a convergence  exponent then series \eqref{1} may or  may not be convergent. For instance, the sequences $a_{n}= 1/n^{\lambda}$ and $1/(n\ln^{2} n)^{\lambda}$ have the same convergence exponent $\lambda=1,$ but in the first case the series \eqref{1} is divergent when $\lambda=1$ while in the second one it is convergent. In this paper, we have a special interest regarding the first case. Consider the following obvious relation between the convergence exponent $\rho$ and the genus $p$ of the corresponding canonical product $p \leq\rho\leq p+1.$ It is clear that if $\rho=p,$   then    the series  \eqref{1} diverges for $\lambda=\rho,$ while $\rho=p+1$ means that the series converges  (in accordance with the definition of $p$). In the monograph \cite{firstab_lit:Eb. Levin},   a more precise characteristic of the density of the sequence $\{a_{n}\}_{1}^{\infty}$ is considered than the convergence exponent. Thus, there is defined  a so-called counting function $n(r)$ equals to a number of points of the sequence in the circle $|z|<r.$ By upper density of the sequence, we call a number
 $$
 \Delta=\overline{\lim\limits_{r\rightarrow\infty}} n(r)/r^{\rho}.
 $$
 If a limit exists in the ordinary sense (not in the  sense of the upper limit),  then $\Delta$ is called the density. Note that it is proved in Lemma 1
  \cite{firstab_lit:Eb. Levin} that
 $$
   \lim\limits_{r\rightarrow\infty}  n(r)/r^{\rho+\varepsilon}\rightarrow 0,\,\varepsilon>0.
 $$
Bellow, we refer to the Theorem 13 \cite{firstab_lit:Eb. Levin} (Chapter I, $\S$ 10) that gives us a representation of the entire function of the finite order. To avoid the any sort of inconvenient form of writing, we will also  call by a root a zero of the entire function.
\begin{teo}\label{T1a} The entire function $f$ of the finite order $\varrho$ has the following representation
$$
f(z)=z^{m}e^{P(z)}\prod\limits_{n=1}^{\omega}G\left(\frac{z}{a_{n}};p\right),\;\omega\leq \infty,
$$
where $a_{n}$ are non-zero roots of the entire function, $p\leq\varrho,\;P(z)$ is a polynomial, $\mathrm{deg}\, P(z)\leq \varrho,\;m$ is a multiplicity of the zero root.
\end{teo}
The infinite product represented  in   Theorem \ref{T1a} is called  a  canonical product of the entire function.\\

\noindent{\bf Proximate order and angular density of zeros}\\

The scale of the growth admits   further clarifications. As a simplest generalization E.L.  Lindel\"{o}f made a comparison $M_{f}(r)$ with the functions of the type
$$
r^{\varrho}\ln^{\alpha_{1}}r \ln^{\alpha_{2}}_{2}r...\ln^{\alpha_{n}}_{n}r,
$$
where $\ln _{j}r=\ln\ln_{j-1}r,\;\alpha_{j}\in \mathbb{R},\,j=1,2,...,n.$ In order to make the further generalization, it is natural (see \cite{firstab_lit:Eb. Levin}) to define a class of the functions $L(r)$ having  {\it low growth}  and  compare $\ln M_{f}(r)$  with $r^{\varrho}L(r).$ Following the idea, G. Valiron introduced a notion of proximate order of the growth of the entire function $f,$ in accordance with which a function $\varrho(r),$  satisfying the following conditions
$$
\lim\limits_{r\rightarrow \infty}\varrho(r)=\varrho;\,\lim\limits_{r\rightarrow \infty}r\varrho'(r)\ln r=0,
$$
is said to be proximate order if the following relation holds
$$
\sigma_{f}=\overline{\lim \limits_{r\rightarrow \infty}} \,\frac{\ln M_{f}(r)}{r^{\varrho(r)}},\,0<\sigma_{f}<\infty.
$$
In this case the value $\sigma_{f}$ is said to be a type of the function $f$ under the proximate order $\varrho(r).$\\

To guaranty some technical results we need to consider a class of  entire functions  whose zero distributions have a certain type of regularity. We follow the monograph  \cite{firstab_lit:Eb. Levin} where the regularity of the distribution of the zeros is characterized by a certain type of density of the set of zeros.

We will say that the set $\Omega$ of the complex plane has an {\it angular density of index} $$\;\xi(r)\rightarrow\xi,\,r\rightarrow\infty,$$
if for an arbitrary set of values $\phi$ and $\psi\;(0<\phi<\psi\leq 2\pi),$ maybe except of   denumerable sets, there exists the limit
\begin{equation}
\Delta(\phi,\psi)=\lim\limits_{r\rightarrow \infty}\frac{n(r,\phi,\psi)}{r^{\xi(r)}},
\end{equation}
where $n(r,\phi,\psi)$ is the number of points of the set $\Omega$ within the sector $|z|\leq r,\;\phi< \mathrm{arg} z<\psi.$ The quantity $\Delta(\phi,\psi)$ will be called the angular density of the set  $\Omega$ within the sector $\phi< \mathrm{arg} z<\psi.$   For a fixed $\phi,$ the relation
$$
\Delta(\psi)-\Delta(\phi)=\Delta(\phi,\psi)
$$
determines, within the additive constant, a nondecreasing function $\Delta(\psi).$  This function is defined for all values of $\psi,$ may be except for a denumerable set of values. It is shown in the monograph \cite[p. 89]{firstab_lit:Eb. Levin}  that the exceptional values of $\phi$ and $\psi$ for which there does  not exist an angular density must be the points of discontinuity of the function $\Delta(\psi).$  A set will be said to be {\it regularly  distributed} relative to $\xi(r)$  if it has an angular density $\xi(r)$ with $\xi$ non-integer.

The asymptotic equalities which we will establish are related to the order of growth. By the asymptotic equation
$$
f(r)\approx \varphi(r)
$$
we will mean the fulfilment of the following condition
$$
[f(r)-\varphi(r)]/r^{\varrho(r)}\rightarrow0,\,r\rightarrow\infty.
$$

Consider the following conditions allowing us to solve technical problems related to estimation of contour integrals.    \\

\noindent $(\mathrm{I})$ There exists a value $d>0$ such that circles of  radii
$$
r_{n}=d|a_{n}|^{1-\frac{\varrho(|a_{n}|)}{2}}
$$
with the centers situated at the points $a_{n}$ do not intersect each other, where $a_{n}.$\\

\noindent $(\mathrm{II})$ The points $a_{n}$ lie inside angles with a common vertex at the origin but with no other points in common, which are such that if one arranges the points of the set $\{a_{n}\}$ within any one of these angles in the order of increasing moduli, then for all points which lie inside the same angle the following relation holds
$$
|a_{n+1}|-|a_{n}|>d|a_{n}|^{1-\varrho(|a_{n}|)},\,d>0.
$$
 The circles $|z-a_{n}|\leq r_{n}$ in the first case, and $|z-a_{n}|\leq d |a_{n}|^{1-\varrho(|a_{n}|)}$ in the second case, will be called the exceptional  circles.\\

The following theorem is a central point of the study. Bellow for the reader convenience, we present   the  Theorem 5 \cite{firstab_lit:Eb. Levin} (Chapter II, $\S$ 1) in the slightly changed form.

\begin{teo}\label{T2a} Assume that the entire  function $f$  of the proximate order $\varrho(r),$    where $\varrho$ is not integer, is represented by its canonical product i.e.
$$
f(z)= \prod\limits_{n=1}^{\infty}G\left(\frac{z}{a_{n}};p\right),
$$
the set of zeros is regularly distributed relative to the proximate order and  satisfies one of the conditions $(\mathrm{I})$ or $(\mathrm{II}).$    Then outside of the exceptional set of circulus   the entire  function  satisfies the following  asymptotical inequality
$$
\ln |f(re^{i\psi})|\approx H(\psi)r^{\varrho(r)},
$$
where
$$
H(\psi):=\frac{\pi}{\sin \pi \varrho}\int\limits_{\psi-2 \pi}^{\psi}   \cos \varrho (\psi-\varphi-\pi)  d\Delta(\varphi).
$$
\end{teo}
The following lemma gives us a key for the technical part of being constructed theory. Although it does not contain implications of any subtle sort it is worth being presented in the expanded form for the reader convenience.
\begin{lem}\label{L1a}
Assume that $ \varrho\in (0,1/2]$ then the  function $H(\psi)$ is positive if $ \psi\in(-\pi,\pi).$
\end{lem}
\begin{proof}
Taking  into account the facts $\cos \varrho (\psi-\varphi-\pi)=\cos \varrho (|\psi-\varphi|-\pi),\,\psi-2\pi<\varphi<\psi,\;\;\cos \varrho (|\psi-\varphi|-\pi)=\cos \varrho (|\psi-(\varphi+2\pi)|-\pi),$ we obtain the following form
$$
H(\psi):=\frac{\pi}{\sin \pi \varrho}\int\limits_{0}^{2 \pi}   \cos \varrho (|\psi-\varphi|-\pi)  d\Delta(\varphi).
$$
Having noticed the following correspondence between sets $\varphi\in [0,\psi]\Rightarrow\xi\in [\varrho(\psi-\pi),-\varrho\pi],$ $\varphi\in [\psi,\psi+\pi]\Rightarrow\xi\in [ -\varrho\pi,0],$ $\varphi\in [\psi+\pi, 2\pi]\Rightarrow\xi\in [  0,\varrho(\pi-\psi)],$ where $\xi:=\varrho (|\psi-\varphi|-\pi),$ we conclude that $\cos \varrho (|\psi-\varphi|-\pi) \geq 0,\,\varphi\in [0,2\pi].$ Taking into account the fact that the function $\Delta(\varphi)$ is non-decreasing, we obtain the desired result.
\end{proof}

\section{Main results}

 In this section, we have a challenge how to generalize results \cite{firstab_lit:1kukushkin2021} in the way to make an efficient tool for study abstract fractional evolution equations with the operator function in the second term. The operator function is supposed to be defined on the set of unbounded non-selfadjoint operators.
  First of all,  we consider   statements
with the  necessary  refinement caused by  the involved functions, here we should note that a particular case corresponding to a power function $\varphi$   was considered by Lidskii \cite{firstab_lit:1Lidskii}.  Secondly, we find conditions that guarantee convergence of the involved integral constructions and formulate lemmas giving us a tool for further study.   As  a main result, we prove an existence and uniqueness theorem for an abstract  fractional evolution equation with the operator function  in the second term. Finally, we discuss  an approach   that  we can implement  to apply  the abstract theoretical results  to  concrete  evolution equations.\\

\noindent{\bf Estimate  of a real component from below}\\

In this subsection we aim to produce estimates of the real component from bellow for the technical  purposes formulated in the further paragraphs.  We should admit that it is formulated in rather rough manner but its principal value is the discovered way of constructing entire functions fallen in the scope of the theory of fractional evolution equations with the operator function in the second term. Apparently, having put a base, we can weaken conditions imposed upon the entire functions class afterwards and in this way come to the natural theory. We need involve some technicalities related to the estimates of the entire unction  from bellow, we should remind that this matter is very important in the constructed theory. We consider a sector $\mathfrak{L}_{0}(\theta_{0},\theta_{1}):=\{z\in \mathbb{C},\, \theta_{0}\leq arg z \leq\theta_{1}\}$ and use a short-hand notation $\mathfrak{L}_{0}(\theta):=\mathfrak{L}_{0}(-\theta,\theta).$

\begin{lem}\label{L2a} Assume that the entire function $f$ is of the proximate   order $\varrho(r),\,\varrho\in (0,1/2],$    maps the ray  $\mathrm{arg}\,z=\theta_{0}$   within  a sector $\mathfrak{L}_{0}(\zeta),\,0<\zeta<\pi/2,$   the set of zeros is regularly distributed relative to the proximate order and  satisfies one of the conditions $(\mathrm{I})$ or $(\mathrm{II}),$   there exists $\varepsilon>0$ such that  the    angle  $ \theta_{0}-\varepsilon<arg\,z  <\theta_{0}+\varepsilon  $    do not contain the zeros  with the sufficiently large absolute value.
 Then,  for a sufficiently large value $r,$ the following relation holds
$$
\mathrm{Re} f(z)>   Ce^{ H(\theta_{0})r^{\varrho(r)}}, \mathrm{arg}\,z=\theta_{0}.
$$
\end{lem}
\begin{proof} Using Theorem \ref{T1a}, we obtain the following representation
$$
f(z)=Cz^{m} \prod\limits_{n=1}^{\infty}G\left(\frac{z}{a_{n}};p\right) ,
$$
here we should remark that    $\mathrm{deg} P(z)=0.$   Let us show that the proximate order of the canonical product of the entire function is the same, we have
$$
M_{f}(r)=Cr^{m} M_{F}(r),\,F(z)=\prod\limits_{n=1}^{\infty}G\left(\frac{z}{a_{n}};p\right).
$$
Therefore in accordance with the definition of proximate order, we have
$$
 \overline{\lim \limits_{r\rightarrow \infty}}  \left\{\frac{m\ln r+\ln C }{r^{\varrho(r)}}  +\frac{\ln M_{F}(r)}{ r^{\varrho(r)}}\right\}  =\sigma_{f},\;0<\sigma_{f}<\infty,
$$
from what follows easily the fact $0<\sigma_{F}<\infty,$ moreover $\sigma_{F}=\sigma_{f}.$
 Note that due to the condition that guarantees that the image of the ray  $\mathrm{arg}\,z=\theta_{0}$ belongs to a sector in the right half-plane, we get
$$
\mathrm{Re}f(z)\geq (1+\tan \zeta)^{-1/2} |f(z)|,\,r=|z|,\; \mathrm{arg}\,z=\theta_{0}.
$$
Applying Theorem \ref{T2a}  we conclude, that excluding  the intersection of the  exceptional set of circulus with the ray  $ \mathrm{arg}\,z=\theta_{0},$  the following relation holds for sufficiently large values $r$
$$
|f(z)|=C r^{m} \left|\prod\limits_{n=1}^{\infty}G\left(\frac{z}{a_{n}};p\right)\right|\geq Cr^{m}e^{H(\theta_{0})r^{\varrho(r)}},
$$
 where $H(\theta_{0})>0$ in accordance with Lemma \ref{L1a}.  It is clear that if we show that the   intersection of the ray  $\mathrm{arg}\,z=\theta_{0}$ with the exceptional set of circulus  is empty, then we complete the proof. Note that the character of the zeros distribution   allows us to claim that is true. In accordance with the lemma conditions,  it suffices to  consider the neighborhoods  of the zeros defined as follows    $|z-a_{n}|< d |a_{n}|^{1-\varrho(|a_{n}|)},\,|z-a_{n}|< d |a_{n}|^{1-\varrho(|a_{n}|)/2}$ and note that  $0< \varrho(|a_{n}|)<1$ for a sufficiently large number $n\in \mathbb{N},$ since
 $
  \varrho(|a_{n}|)\rightarrow\varrho,\,n\rightarrow\infty.
 $
 Here, we ought to remind that the zeros are arranged in order with their absolute value growth.
Thus, using simple properties of the power function with the positive exponent less than one, we obtain the fact that the intersection of the  exceptional set of circulus with the ray  $ \mathrm{arg}\,z=\theta_{0}$ is empty for a sufficiently large $n\in \mathbb{N}.$
\end{proof}

\noindent{\bf Classical  lemmas in the refined form }\\

Bellow, we consider an invertible operator $B$ with a dense range  and use a notation $W:=B^{-1}.$ This agrement is justified by the significance of the operator with a compact resolvent, the detailed information on which spectral properties can be found in the papers   cited in the introduction section.
Consider an entire function    $\varphi$ that can be represented by its Taylor series about the point zero.  Denote by
\begin{equation}\label{12a}
\varphi(W):=\sum\limits_{n=0}^{\infty}  c_{n}  W^{n}
\end{equation}
a formal construction called by an operator function, where $c_{n}$ are the Taylor coefficients corresponding to the function $\varphi.$
The lemma  given bellow are devoted to the study of the conditions under which being imposed the series of operators \eqref{12a} converges on some elements of the Hilbert space $\mathfrak{H},$   thus the operator $\varphi(W)$ is defined. Assume that  a compact operator $T: \mathfrak{H}\rightarrow \mathfrak{H}$ is such that
$
\overline{\Theta(T)}\subset \mathfrak{L}_{0}(\theta_{0},\theta_{1} ).
$
Then,  we put the following contour   in correspondence to the operator
\begin{equation*}
\vartheta(T):=\left\{\lambda:\;|\lambda|=r>0,\, \theta_{0} \leq\mathrm{arg} \lambda \leq \theta_{1} \right\}\cup\left\{\lambda:\;|\lambda|>r,\;
  \mathrm{arg} \lambda =\theta_{0} ,\,\mathrm{arg} \lambda =\theta_{1} \right\},
\end{equation*}
where   the number $r$ is chosen so that the operator  $ (I-\lambda T)^{-1} $ is regular within the corresponding closed circle.
Bellow,  we consider the following hypophyses separately written for convenience of the reader.   \\

\noindent $(\mathrm{H}\mathrm{I})$ The operator $B$ is compact, $\overline{\Theta(B)}\subset \mathfrak{L}_{0}(\theta_{0},\theta_{1} ),$   the entire function $\varphi$ of the order less than a half  maps the sector $\mathfrak{L}_{0}(\theta_{0},\theta_{1} )$ into the sector $\mathfrak{L}_{0}(\varpi),\,\varpi<\pi/2\alpha,\,\alpha>0,$  its zeros with a sufficiently large absolute value   do not belong to the sector $\mathfrak{L}_{0}(\theta_{0},\theta_{1} ).$

\begin{lem}\label{L3a}  Assume that the condition $(\mathrm{H}\mathrm{I})$ holds,       the entire function $\varphi$ is of the order less than a half.  Then the following relation holds
$$
 \int\limits_{\vartheta(B)}\varphi(\lambda) e^{-\varphi^{\alpha}(\lambda)  t} B(I-\lambda B)^{-1}fd\lambda= \varphi(W)\!\!\!\int\limits_{\vartheta(B)}e^{-\varphi^{\alpha}(\lambda) t}B(I-\lambda B)^{-1}fd\lambda,\,f\in \mathrm{D}(W^{n}),\,\forall n\in \mathbb{N},
$$
moreover
$$
\lim\limits_{t\rightarrow+0}\frac{1}{2 \pi i}\int\limits_{\vartheta(B)}e^{-\varphi^{\alpha}(\lambda) t}B(I-\lambda B)^{-1}fd\lambda=f,\,f\in \mathrm{D}(W).
$$
\end{lem}
\begin{proof} Firstly, we should note that the conditions imposed upon the order of the function $\varphi$ alow us to claim that the latter integral converges for a fixed value of the parameter $t.$
Let us establish the formula
\begin{equation}\label{6y}
 \int\limits_{\vartheta(B)}\varphi(\lambda) e^{-\varphi^{\alpha}(\lambda)  t} B(I-\lambda B)^{-1}f d\lambda=\sum\limits_{n=0}^{\infty}c_{n} \!\!
 \int\limits_{\vartheta(B)} e^{-\varphi^{\alpha}(\lambda)  t} \lambda^{n}B(I-\lambda B)^{-1} fd\lambda.
\end{equation}
To prove this fact, we should show  that the following relation holds
\begin{equation}\label{7y}
 \int\limits_{\vartheta_{j}  (B)}\!\!\varphi(\lambda) e^{-\varphi^{\alpha}(\lambda)  t} B(I-\lambda B)^{-1}f d\lambda=\sum\limits_{n=0}^{\infty}c_{n} \!\! \int\limits_{\vartheta_{j}  (B)} e^{-\varphi^{\alpha}(\lambda)  t} \lambda^{n}B(I-\lambda B)^{-1} fd\lambda,
\end{equation}
where
$$
\vartheta_{j}(B):=\left\{\lambda:\;|\lambda|=r>0,\, \theta_{0} \leq\mathrm{arg} \lambda \leq \theta_{1} \right\}\cup\left\{\lambda:\;r<|\lambda|<r_{j},\,\mathrm{arg} \lambda =\theta_{0}   ,\,\mathrm{arg} \lambda =\theta_{1} \right\},
$$
$r_{j}\uparrow \infty.$
Note that in accordance with  Lemma 6 \cite{firstab_lit:1kukushkin2021}, we get
$$
\|(I-\lambda B)^{-1}\|\leq C,\,r<|\lambda|<r_{j},\,\mathrm{arg} \lambda =\theta_{0}   ,\,\mathrm{arg} \lambda =\theta_{1} .
$$
Using this estimate, we can easily obtain the fact
$$
\sum\limits_{n=0}^{\infty}|c_{n}|    | e^{-\varphi^{\alpha}(\lambda)  t}| |\lambda^{n}|\cdot\|B(I-\lambda B)^{-1} f\|  \leq C\|B\|\cdot \|f\| \sum\limits_{n=0}^{\infty}    |c_{n}|  |\lambda|^{n}      e^{- \mathrm{Re}\varphi^{\alpha}(\lambda)  t},\,\lambda\in \vartheta_{j}(B),
$$
where the latter series is convergent. Therefore, reformulating      the well-known theorem of calculus  on the absolutely convergent series in   terms of the norm, we obtain   \eqref{7y}. Now, let us show that the series
\begin{equation}\label{8y}
\sum\limits_{n=0}^{\infty}c_{n}\!\!\!\int\limits_{ \vartheta_{j}  (B)}  e^{-\varphi^{\alpha}(\lambda)  t} \lambda^{n}B(I-\lambda B)^{-1} fd\lambda
\end{equation}
is uniformly convergent with respect to $j\in \mathbb{N}.$
 Using Lemma 1 \cite{firstab_lit(axi2022)}, we get  a trivial inequality
$$
 \left\|\,\int\limits_{\vartheta_{j}  (B)} e^{-\varphi^{\alpha}(\lambda)  t} \lambda^{n}B(I-\lambda B)^{-1} fd\lambda \right\|_{\mathfrak{H}}\leq C\|f\|_{\mathfrak{H}} \!\!\!\int\limits_{\vartheta_{j}  (B)} e^{-\mathrm{Re} \varphi^{\alpha}(\lambda)  t  } |\lambda|^{n} |d\lambda|\leq
 $$
 $$
 \leq C\|f\|_{\mathfrak{H}} \int\limits_{\vartheta_{j}  (B)} e^{-C|\varphi(\lambda)|^{\alpha}t} |\lambda|^{n} |d\lambda|.
 $$
Here, we should note that to obtain the desired result one is satisfied with a rather rough estimate dictated by the estimate obtained in Lemma \ref{L2a}, we get
$$
 \int\limits_{\vartheta_{j}  (B)} e^{-| \varphi(\lambda)|^{\alpha}  t  } |\lambda|^{n} |d\lambda|\leq    C\int\limits_{r}^{r_{j}} e^{-  x  t  } x^{n}  dx\leq C t^{-n}\Gamma(n+1) .
$$
Thus, we obtain
\begin{equation*}
\left\|\,\int\limits_{ \vartheta_{j}(B)} e^{-\varphi^{\alpha}(\lambda)  t} \lambda^{n}B(I-\lambda B)^{-1} fd\lambda \right\|_{\mathfrak{H}}\leq C t^{-n}n! \,.
\end{equation*}
Using the standard  formula   establishing the estimate for the  Taylor coefficients  of   the entire function, then applying    the Stirling formula, we get
 $$
|c_{n}|<\left(e \sigma \varrho \right)^{n/\varrho} n^{-n/\varrho}< (2\pi)^{1/2\varrho}(\sigma\varrho)^{n/\varrho}\left(\frac{\sqrt{n}}{n!}\right)^{1/\varrho}\!\!\!,
$$
where $0<\sigma<\infty$ is a type of the function $\varphi.$ Thus, we obtain
\begin{equation*}
 \sum\limits_{n=1}^{\infty}|c_{n}|\left\|\,\int\limits_{\vartheta_{j}(B)} e^{-\varphi^{\alpha}(\lambda)  t} \lambda^{n}B(I-\lambda B)^{-1} fd\lambda\right\|\leq  C\sum\limits_{n=1}^{\infty}(\sigma\varrho)^{n/\varrho}t^{-n}(n!)^{1-1/\varrho} n^{1/2\varrho}  .
\end{equation*}
The latter series is convergent for an arbitrary fixed  $t>0,$ what proves the  uniform convergence of the series \eqref{8y} with respect to $j\,.$ Therefore, reformulating      the well-known theorem of calculus applicably to the norm of the Hilbert space,  taking into accounts the facts
\begin{equation*}
 \int\limits_{\vartheta_{j} (B)}\!\!\varphi(\lambda) e^{-\varphi^{\alpha}(\lambda)  t} B(I-\lambda B)^{-1}f d\lambda \stackrel{\mathfrak{H}}{\longrightarrow} \!\int\limits_{\vartheta   (B)}\!\!\varphi(\lambda) e^{-\varphi^{\alpha}(\lambda)  t} B(I-\lambda B)^{-1}f d\lambda,
 $$
 $$
 \int\limits_{\vartheta_{j}  (B)}\!\! e^{-\varphi^{\alpha}(\lambda)  t} \lambda^{n}B(I-\lambda B)^{-1} fd\lambda\stackrel{\mathfrak{H}}{\longrightarrow}\!\int\limits_{\vartheta   (B)} e^{-\varphi^{\alpha}(\lambda)  t} \lambda^{n}B(I-\lambda B)^{-1} fd\lambda,\;j\rightarrow\infty,
\end{equation*}
we obtain  formula \eqref{6y}. Further, using the formula
\begin{equation*}
  \lambda^{k} B^{k}(I-\lambda B)^{-1}=(I-\lambda B)^{-1}-(I+\lambda B+...+\lambda^{k-1}B^{k-1}),\;k\in \mathbb{N},
\end{equation*}
taking into account the facts that the operators  $B^{k}$ and $(I-\lambda B)^{-1}$ commute,
we obtain
$$
 \int\limits_{\vartheta(B)}e^{-\varphi^{\alpha}(\lambda)  t}\lambda^{n}B(I-\lambda B)^{-1} fd\lambda=
$$
$$
= \int\limits_{\vartheta(B)}e^{-\varphi^{\alpha}(\lambda)  t}B(I-\lambda B)^{-1}W^{n}fd\lambda- \int\limits_{\vartheta(B)}e^{-\varphi^{\alpha}(\lambda)  t} \sum\limits_{k=0}^{n-1}\lambda^{k}B^{k+1} W^{n}fd\lambda=I_{1}(t)+I_{2}(t).
$$
Since the operators $W^{n}$ and $B(I-\lambda B)^{-1}$ commute,   this fact can be obtained  by direct calculation, we get
$$
I_{1}(t)=W^{n}\!\!\!\int\limits_{\vartheta(B)}e^{-\varphi^{\alpha}(\lambda)  t}B(I-\lambda B)^{-1} fd\lambda.
$$
Consider $I_{2}(t),$ using the technique applied in Lemma 5 \cite{firstab_lit(axi2022)} it is rather reasonable to consider the following representation
\begin{equation*}
 I_{2}(t):=
 -\sum\limits_{k=0}^{n-1 }\beta_{k}(t)B^{k-n+1}  f,\;
   \beta_{k}(t):=\!\!\int\limits_{\vartheta(B)}\!\!e^{-\varphi^{\alpha}(\lambda)t} \lambda^{k}d\lambda.
\end{equation*}
Analogously to the scheme of   reasonings of Lemma 5 \cite{firstab_lit(axi2022)}, we can show that $\beta_{k}(t)=0,$ under  the imposed  condition of the entire function growth regularity. Bellow, we produce a complete reasoning to avoid any kind of   misunderstanding.
Let us show that  $\beta_{k}(t)=0,$  define a contour  $\vartheta_{R}(B):= \mathrm{Fr}\left\{\mathrm{int }\,\vartheta(B) \,\cap \{\lambda:\,r<|\lambda|<R \}\right\}$ and let us prove that there exists such a sequence    $\{R_{n}\}_{1}^{\infty},\,R_{n}\uparrow \infty$ that
\begin{equation}\label{17a}
  \oint\limits_{\vartheta_{R_{n}}(B)}\!\!\!\!e^{-\varphi^{\alpha}(\lambda)  t} \lambda^{k}d\lambda\rightarrow \beta_{k}(t),\;n\rightarrow \infty.
\end{equation}
 Consider a decomposition of the contour $\vartheta_{R}(B)$ on terms
$
\tilde{\vartheta}_{ R}:=\{\lambda:\,|\lambda|=R,\, \theta_{0}  \leq\mathrm{arg} \lambda  \leq\theta_{1}  \},
$
and
$
\hat{\vartheta}_{R}:= \{\lambda:\,|\lambda|=r,\,\theta_{0}  \leq\mathrm{arg} \lambda  \leq\theta_{1}\}\cup
\{\lambda:\,r<|\lambda|<R,\, \mathrm{arg} \lambda  =\theta_{0} ,\,\mathrm{arg} \lambda  = \theta_{1} \}.
$
We have
$$
 \oint\limits_{\vartheta_{R}(B)}e^{-\varphi^{\alpha}(\lambda)  t} \lambda^{k}d\lambda=
 \int\limits_{\tilde{\vartheta}_{ R}}e^{-\varphi^{\alpha}(\lambda)  t} \lambda^{k}d\lambda+
  \int\limits_{\hat{\vartheta}_{R}}e^{-\varphi^{\alpha}(\lambda)  t} \lambda^{k}d\lambda.
$$
Having noticed that the integral at the left-hand side of the last relation   equals to zero, since the   function under the integral is analytic   inside the contour, we  come to the conclusion that to obtain the desired result it suffices to show that
\begin{equation}\label{18a}
 \int\limits_{\tilde{\vartheta}_{ R_{n}}}e^{-\varphi^{\alpha}(\lambda)  t} \lambda^{k}d\lambda\rightarrow 0,
\end{equation}
where     $\{R_{n}\}_{1}^{\infty},\,R_{n}\uparrow \infty.$
We have
\begin{equation*}
  \left|\,\int\limits_{\tilde{\vartheta}_{ R}}e^{-\varphi^{\alpha}(\lambda)  t}\lambda^{k}    d \lambda\,\right| \leq  R^{k}\int\limits_{\tilde{\vartheta}_{ R}}|e^{-\varphi^{\alpha}(\lambda)  t}| |d \lambda|\leq R^{k+1}\int\limits_{ \theta_{0} }^{\theta_{1}}  e^{-\mathrm{Re} \varphi^{\alpha}(\lambda)  t  } d \,\mathrm{arg} \lambda.
\end{equation*}
Consider a value   $\mathrm{Re}\,\varphi^{\alpha}(\lambda),\, \lambda\in  \tilde{\vartheta}_{ R}$  for a sufficiently large value $R.$
Using the condition imposed upon the order of the entire function and  applying  the Wieman theorem  (Theorem 30,  \S 18, Chapter I \cite{firstab_lit:Eb. Levin}),  we can claim that there exists such a sequence $\{R_{n}\}_{1}^{\infty},\,R_{n}\uparrow \infty$ that
\begin{equation*}
\forall \varepsilon>0,\,\exists N(\varepsilon):\,e^{- C|\varphi(\lambda )|^{\alpha}t}\leq e^{- C m^{\alpha}_{\varphi}(R_{n})t}\leq e^{- C t[M_{\varphi}(R_{n})]^{(\cos \pi \varrho-\varepsilon)\alpha}},\,\lambda \in \tilde{\vartheta}_{ R_{n}},\,n>N(\varepsilon),
\end{equation*}
where $\varrho$ is the order of the entire  function $\varphi.$ Applying this estimate, we obtain
$$
\int\limits_{ \theta_{0} }^{\theta_{1}}  e^{-\mathrm{Re} \varphi^{\alpha}(\lambda)  t  } d \,\mathrm{arg} \lambda \leq
  \int\limits_{ \theta_{0} }^{\theta_{1}}  e^{- C t |\varphi (\lambda)|^{\alpha} }  d \,\mathrm{arg} \lambda\leq   e^{- C t[M_{\varphi}(R_{n})]^{(\cos \pi \varrho-\varepsilon)\alpha}}\!\!\!\int\limits_{ \theta_{0} }^{\theta_{1}}    d \,\mathrm{arg} \lambda.
$$
The latter estimate gives us \eqref{18a} from what follows \eqref{17a}. Therefore $\beta_{k}(t)=0,$ hence $I_{2}(t)=0$ and we get
$$
 \int\limits_{\vartheta(B)}e^{-\varphi^{\alpha}(\lambda)  t}\lambda^{n}B(I-\lambda B)^{-1} fd\lambda=W^{n}\!\!\!\!\int\limits_{\vartheta(B)}e^{-\varphi^{\alpha}(\lambda) t}B(I-\lambda B)^{-1}fd\lambda.
$$
Substituting the latter relation into the formula \eqref{6y}, we obtain the first statement of the lemma.

The scheme of the proof corresponding to the second statement  is absolutely analogous to the one presented in Lemma 4 \cite{firstab_lit(axi2022)}, we should just use Lemma \ref{L2a} providing the estimates along the sides of the contour. Thus, the completion of the reasonings is due to the technical repetition of the Lemma 4 \cite{firstab_lit(axi2022)} reasonings, we left it to the reader.
\end{proof}

\noindent{\bf   Series expansion  and its application to the  Existence and uniqueness theorems}\\

This paragraph is a climax of the paper, here we represent a theorem that put a beginning of a marvelous research based on the Abell-Lidskii method. The attempt to consider an operator function in the second term  was made in the paper  \cite{firstab_lit:2kukushkin2022}, where we consider a case that is not so difficult  since  the operator function was represented by a Laurent series with a polynomial regular part. In contrast, in this paper we consider a more complicated case, a function that compels us to involve a principally different method of study. The existence and uniqueness theorem given bellow is based upon  one of  the  theorems  represented in \cite{firstab_lit:1kukushkin2021}.

  We can choose a  Jordan basis in   the finite dimensional root subspace $\mathfrak{N}_{q}$ corresponding to the eigenvalue $\mu_{q}$  that consists of Jordan chains of eigenvectors and root vectors  of the operator $B_{q}.$  Each chain has a set of numbers
\begin{equation}\label{12i}
  e_{q_{\xi}},e_{q_{\xi}+1},...,e_{q_{\xi}+k},\,k\in \mathbb{N}_{0},
\end{equation}
  where $e_{q_{\xi}},\,\xi=1,2,...,m $ are the eigenvectors  corresponding   to the  eigenvalue $\mu_{q}.$
Considering the sequence $\{\mu_{q}\}_{1}^{\infty}$ of the eigenvalues of the operator $B$ and choosing a  Jordan basis in each corresponding  space $\mathfrak{N}_{q},$ we can arrange a system of vectors $\{e_{i}\}_{1}^{\infty}$ which we will call a system of the root vectors or following  Lidskii  a system of the major vectors of the operator $B.$
Assume that  $e_{1},e_{2},...,e_{n_{q}}$ is  the Jordan basis in the subspace $\mathfrak{N}_{q}.$  We can prove easily (see \cite[p.14]{firstab_lit:1Lidskii}) that     there exists a  corresponding biorthogonal basis $g_{1},g_{2},...,g_{n_{q}}$ in the subspace $\mathfrak{M}_{q}^{\perp},$ where  $\mathfrak{M}_{q},$ is a subspace  wherein the operator  $B-\mu_{q} I$ is invertible. Using the reasonings \cite{firstab_lit:1kukushkin2021},  we conclude that $\{ g_{i}\}_{1}^{n_{q}}$ consists of the Jordan chains of the operator $B^{\ast}$ which correspond to the Jordan chains  \eqref{12i}, more detailed information can be found in \cite{firstab_lit:1kukushkin2021}.
It is not hard to prove   that  the set  $\{g_{\nu}\}^{n_{j}}_{1},\,j\neq i$  is orthogonal to the set $ \{e_{\nu}\}_{1}^{n_{i}}$ (see \cite{firstab_lit:1kukushkin2021}).  Gathering the sets $\{g_{\nu}\}^{n_{j}}_{1},\,j=1,2,...,$ we can obviously create a biorthogonal system $\{g_{n}\}_{1}^{\infty}$ with respect to the system of the major vectors of the operator $B.$  Consider a sum
\begin{equation}\label{3h}
 \mathcal{A} _{\nu}(\varphi,t)f:= \sum\limits_{q=N_{\nu}+1}^{N_{\nu+1}}\sum\limits_{\xi=1}^{m(q)}\sum\limits_{i=0}^{k(q_{\xi})}e_{q_{\xi}+i}c_{q_{\xi}+i}(t)
\end{equation}
where   $k(q_{\xi})+1$ is a number of elements in the $q_{\xi}$-th Jourdan chain,  $m(q)$ is a geometrical multiplicity of the $q$-th eigenvalue,
\begin{equation}\label{4h}
c_{q_{\xi}+i}(t)=   e^{ -\varphi(\lambda_{q})  t}\sum\limits_{m=0}^{k(q_{\xi})-i}H_{m}(\varphi, \lambda_{q},t)c_{q_{\xi}+i+m},\,i=0,1,2,...,k(q_{\xi}),
\end{equation}
$
c_{q_{\xi}+i}= (f,g_{q_{\xi}+k-i}) /(e_{q_{\xi}+i},g_{q_{\xi}+k-i}),
$
$\lambda_{q}=1/\mu_{q}$ is a characteristic number corresponding to $e_{q_{\xi}},$
$$
H_{m}( \varphi,z,t ):=  \frac{e^{ \varphi(z)  t}}{m!} \cdot\lim\limits_{\zeta\rightarrow 1/z }\frac{d^{m}}{d\zeta^{\,m}}\left\{ e^{-\varphi (\zeta^{-1})t}\right\} ,\;m=0,1,2,...\,,\,.
$$
More detailed information on the considered above   Jordan chains can be found in \cite{firstab_lit:1kukushkin2021}.

\begin{teo}\label{T3}
Assume that the condition $(\mathrm{H}\mathrm{I})$ holds,       the entire function $\varphi$ is of the order less than a half, $ B\in \mathfrak{S}_{s},\,0<s<\infty.$
Then a sequence of natural numbers $\{N_{\nu}\}_{0}^{\infty}$ can be chosen so that
 \begin{equation}\label{26}
  \sum\limits_{\nu=0}^{\infty}\left\|\mathcal{A}_{\nu}(\varphi^{\alpha},t)f\right\|_{\mathfrak{H}}<\infty,\,f\in \mathfrak{H};\;\; f=\lim\limits_{t\rightarrow+0} \sum\limits_{\nu=0}^{\infty}\mathcal{A}_{\nu}(\varphi^{\alpha},t)f,\,f\in \mathrm{D}(W).
\end{equation}
 \end{teo}

\begin{proof}
  Let us establish the first  relation \eqref{26}.  Consider a contour $\vartheta(B).$ Having fixed $R>0,0<\kappa<1,$ so that $R(1-\kappa)=r,$ consider a monotonically increasing sequence $\{R_{\nu}\}_{0}^{\infty},\,R_{\nu}=R(1-\kappa)^{-\nu+1}.$   Using Lemma 5 \cite{firstab_lit:1kukushkin2021}, we get
$$
\|(I-\lambda B )^{-1}\|_{\mathfrak{H}}\leq e^{\gamma (|\lambda|)|\lambda|^{\sigma}}|\lambda|^{m},\,\sigma>s,\,m=[\sigma],\,|\lambda|=\tilde{R}_{\nu},\;R_{\nu}<\tilde{R}_{\nu}<R_{\nu+1},
$$
where
$$
\gamma(|\lambda|)= \beta ( |\lambda|^{m+1})  +C \beta(|C  \lambda| ^{m+1}),\;\beta(r )= r^{ -\frac{\sigma}{m+1} }\left(\int\limits_{0}^{r}\frac{n_{B^{m+1}}(t)dt}{t }+
r \int\limits_{r}^{\infty}\frac{n_{B^{m+1}}(t)dt}{t^{ 2  }}\right).
$$
Note that  in accordance with Lemma 3 \cite{firstab_lit:1Lidskii}  the following relation holds
\begin{equation}\label{27}
\sum\limits_{i=1}^{\infty}\lambda^{\frac{\sigma}{ (m+1)}}_{i}( \tilde{B} )\leq \sum\limits_{i=1}^{\infty}s^{ \,\sigma }_{i}( B )<\infty,\,\varepsilon>0,
\end{equation}
where   $\tilde{B}:=(B^{\ast m+1}A^{m+1})^{1/2}.$  It is clear that   $\tilde{B}\in  \mathfrak{S}_{\upsilon},\,\upsilon<\sigma/(m+1).$
Denote by $\vartheta_{\nu}$ a bound of the intersection of the ring $\tilde{R}_{\nu}<|\lambda|<\tilde{R}_{\nu+1}$ with the interior of the contour $\vartheta(B),$ denote by $N_{\nu}$ a number of the resolvent poles  being   contained  in the set $\mathrm{int }\,\vartheta(B) \,\cap \{\lambda:\,r<|\lambda|<\tilde{R}_{\nu} \}.$ In accordance with Lemma 3 \cite{firstab_lit(axi2022)},  we get
\begin{equation}\label{28}
 \frac{1}{2\pi i}\oint\limits_{\vartheta_{\nu}}e^{-\varphi^{\alpha}(\lambda)  t} B(I-\lambda B)^{-1}f d \lambda =\sum\limits_{q=N_{\nu}+1}^{N_{\nu+1}}\sum\limits_{\xi=1}^{m(q)}\sum\limits_{i=0}^{k(q_{\xi})}e_{q_{\xi}+i}c_{q_{\xi}+i}(t),\;f\in \mathfrak{H}.
\end{equation}
Let us estimate the above integral, for this purpose split the contour $\vartheta_{\nu}$ on  terms $\tilde{\vartheta}_{ \nu  }:=\{\lambda:\,|\lambda|=\tilde{R}_{\nu},\, \theta_{0} \leq\mathrm{arg} \lambda  \leq\theta_{1}  \},\,\tilde{\vartheta}_{ \nu+1  },\,\vartheta_{\nu_{-}}:=
\{\lambda:\,\tilde{R}_{\nu}<|\lambda|<\tilde{R}_{\nu+1},\, \mathrm{arg} \lambda  =\theta_{0}  \},\, \vartheta_{\nu_{+}}:=
\{\lambda:\,\tilde{R}_{\nu}<|\lambda|<\tilde{R}_{\nu+1},\, \mathrm{arg} \lambda  =\theta_{1}  \}.$   Applying  the Wieman theorem  (Theorem 30, \S 18, Chapter I \cite{firstab_lit:Eb. Levin}),  we can claim that there exists such a sequence $\{R'_{n}\}_{1}^{\infty},\,R'_{n}\uparrow \infty,\,\tilde{R}_{\nu}<R'_{\nu}<\tilde{R}_{\nu+1}$ that
\begin{equation}\label{16v}
\forall \varepsilon>0,\,\exists N(\varepsilon):\,e^{- C|\varphi(\lambda )|^{\alpha}t}\leq e^{- C m^{\alpha}_{\varphi}(R'_{\nu})t}\leq e^{- C t[M_{\varphi}(R'_{\nu})]^{(\cos \pi \varrho-\varepsilon)\alpha}},\,\lambda \in \tilde{\vartheta}_{ \nu  },\,\nu> N(\varepsilon),
\end{equation} where $\varrho$ is the order of the entire  function $\varphi.$   We should note that the assumption  $\tilde{R}_{\nu}<R'_{n}<\tilde{R}_{\nu+1}$ has been  made without loss of generality of the reasonings for in the context of the proof we do not care on the accurate arrangement of the contours but need to prove the   existence of an arbitrary one. This inconvenience is based upon the uncertainty in the way of chousing the contours in the Wieman theorem, at the same time  at any rate, we can extract a subsequence of    the sequence  $\{\tilde{R}_{n}\}_{1}^{\infty}$ in the way we need. Thus, using the given reasonings, Applying  Lemma 5 \cite{firstab_lit:1kukushkin2021}, relation \eqref{16v},      we get
\begin{equation*}
 J_{  \nu  }: =\left\|\,\int\limits_{\tilde{\vartheta}_{ \nu }}e^{-\varphi^{\alpha}(\lambda)  t} B(I-\lambda B)^{-1}f d \lambda\,\right\|_{\mathfrak{H}}\leq \,
 \int\limits_{\tilde{\vartheta}_{ \nu }}e^{-t\mathrm{Re}\,\varphi^{\alpha}(\lambda)  } \left\|B(I-\lambda B)^{-1}f \right\|_{\mathfrak{H}} |d \lambda|\leq
\end{equation*}
$$
 \leq e^{\gamma (|\lambda|)|\lambda|^{\sigma} }|\lambda|^{m+1}   Ce^{- C t[M_{\varphi}(R'_{\nu})]^{(\cos \pi \varrho-\varepsilon)\alpha}} \int\limits_{-\theta-\varsigma}^{\theta+\varsigma}  d \,\mathrm{arg} \lambda,\,|\lambda|=\tilde{R}_{\nu}.
$$
As a result, we get
$$
J_{ \nu } \leq    e^{\gamma (|\lambda|)|\lambda|^{\sigma} }|\lambda|^{m+1}   Ce^{- C t[M_{\varphi}(R'_{\nu})]^{(\cos \pi \varrho-\varepsilon)\alpha}} ,
   $$
   where
   $
   \,m=[\sigma],\,|\lambda|=\tilde{R}_{\nu}.
$
In accordance with  Lemma 2 \cite{firstab_lit:1kukushkin2021}, we have $\gamma (|\lambda|)\rightarrow 0,\,|\lambda|\rightarrow\infty.$   In accordance with the formula \eqref{1j} we can extract a subsequence from the sequence $\{\tilde{R'}_{n}\}_{1}^{\infty}$ and as a result from the sequence  $\{\tilde{R}_{n}\}_{1}^{\infty}$  so that  for a fixed $t$ and  a sufficiently large $\nu,$ we have  $ \gamma (|\tilde{R}_{\nu}|)|\tilde{R}_{\nu}|^{\sigma} - C t[M_{\varphi}(R'_{\nu})]^{(\cos \pi \varrho-\varepsilon)\alpha}     < 0.$ Here, we have not used a subsequence to not complicate the form of writing.
 Therefore, taking into account the above estimates, we can claim that the following series is convergent
 $$
 \sum\limits_{\nu=0}^{\infty}J_{\nu}<\infty.
$$
  Applying Lemma 6 \cite{firstab_lit:1kukushkin2021}, Lemma \ref{L2a}, we get
$$
 J^{+}_{\nu}: =\left\|\,\int\limits_{\vartheta_{\nu_{+}}}e^{-\varphi^{\alpha}(\lambda)  t} B(I-\lambda B)^{-1}f d \lambda\,\right\|_{\mathfrak{H}}\leq  C\|f\|_{\mathfrak{H}} \cdot C\int\limits_{R_{\nu}}^{R_{\nu+1}}  e^{-  C  t \mathrm{Re}\, \varphi^{\alpha}(\lambda) }   |d   \lambda|\leq
  $$
  $$
\leq C\!\!\!\int\limits_{R_{\nu}}^{R_{\nu+1}}  e^{-  C t |\varphi (\lambda)|^{\alpha} }   |d   \lambda| \leq C  e^{- Cte^{ \alpha H(\theta_{1})R_{\nu}^{\varrho(R_{\nu})}}  }     \int\limits_{R_{\nu}}^{R_{\nu+1}}   |d   \lambda|=
 C  e^{-   C te^{ \alpha H(\theta_{1})R_{\nu}^{\varrho(R_{\nu})}}  }    \{R_{\nu+1}-R_{\nu} \}.
 $$
$$
 J^{-}_{\nu}: =\left\|\,\int\limits_{\vartheta_{\nu_{-}}}e^{-\varphi^{\alpha}(\lambda)  t}B(I-\lambda B)^{-1}f d \lambda\,\right\|_{\mathfrak{H}}\leq   C  e^{-   Cte^{ \alpha H(\theta_{0})R_{\nu}^{\varrho(R_{\nu})}}  }       \int\limits_{R_{\nu}}^{R_{\nu+1}}   |d   \lambda|=
 C  e^{-   Cte^{\alpha  H(\theta_{0})R_{\nu}^{\varrho(R_{\nu})}}  }    \{R_{\nu+1}-R_{\nu} \}.
$$
The obtained results allow us to claim (the proof is omitted) that
$$
  \sum\limits_{\nu=0}^{\infty}J^{+}_{\nu}<\infty,\;\; \sum\limits_{\nu=0}^{\infty}J^{-}_{\nu}<\infty.
$$
Using the formula \eqref{28}, the given above decomposition of the contour $\vartheta_{\nu},$  we   obtain the first relation \eqref{26}. Let us establish the second relation   \eqref{26}, for this purpose, we should note that in accordance with  relation \eqref{28}, the properties of the contour integral, we have
 $$
 \frac{1}{2\pi i}\!\!\oint\limits_{\vartheta_{\tilde{R}_{n}}  (B)}\!\!\!\!e^{-\varphi^{\alpha}(\lambda)  t} B(I-\lambda B)^{-1}f \,d \lambda =
   \sum\limits_{\nu=0}^{n-1} \sum\limits_{q=N_{\nu}+1}^{N_{\nu+1}}\sum\limits_{\xi=1}^{m(q)}
  \sum\limits_{i=0}^{k(q_{\xi})}e_{q_{\xi}+i}c_{q_{\xi}+i}(t)
  ,\;f\in \mathfrak{H},\,n\in \mathbb{N},
$$
where
$
 \vartheta_{R}(B):= \mathrm{Fr}\left\{\mathrm{int }\,\vartheta(B) \,\cap \{\lambda:\,r<|\lambda|<R \}\right\}.
$
Using the proved above  fact $J_{ \nu }\rightarrow0,\,\nu\rightarrow\infty,$    we obtain easily
$$
\frac{1}{2\pi i}\!\!\oint\limits_{\vartheta_{\tilde{R}_{n}}  (B)}\!\!\!\!e^{-\varphi^{\alpha}(\lambda)  t} B(I-\lambda B)^{-1}f \,d \lambda\rightarrow \frac{1}{2\pi i}\oint\limits_{\vartheta  (B)} e^{-\varphi^{\alpha}(\lambda)  t} B(I-\lambda B)^{-1}f \,d \lambda,\;n\rightarrow\infty.
$$
 The latter relation
  gives us the following formula
\begin{equation}\label{19m}
 \frac{1}{2\pi i}\!\!\oint\limits_{\vartheta   (B)}\! e^{-\varphi^{\alpha}(\lambda)  t} B(I-\lambda B)^{-1}f \,d \lambda =
   \sum\limits_{\nu=0}^{\infty} \mathcal{A}_{\nu}(\varphi^{\alpha},t)f,\;f\in \mathfrak{H}.
\end{equation}
If $f\in \mathrm{D}(W),$ then applying Lemma \ref{L3a}, we obtain the second relation \eqref{26}.
\end{proof}

In accordance with the ordinary approach \cite{firstab_lit:2kukushkin2022},   we   consider a Hilbert space $\mathfrak{H}$ consists of   element-functions $u:\mathbb{R}_{+}\rightarrow \mathfrak{H},\,u:=u(t),\,t\geq0$    and we  assume that if $u$ belongs to $\mathfrak{H}$    then the fact  holds for all values of the variable $t.$   We understand such operations as differentiation and integration in the generalized sense what is based on the topological properties  of the Hilbert space $\mathfrak{H}.$ For instance, the derivative is understood as a  limit  in the sense of the norm e.t.c., more detailed information can be found in \cite{firstab_lit:1kukushkin2021}, \cite{firstab_lit:Krasnoselskii M.A.}.  Combining the operations we can consider a generalized  fractional derivative
  in the Riemann-Liouville sense (see \cite{firstab_lit:2kukushkin2022},\cite{firstab_lit:samko1987}),     in the formal form, we have
$$
   \mathfrak{D}^{1/\alpha}_{-}f(t):=-\frac{1}{\Gamma(1-1/\alpha)}\frac{d}{d t}\int\limits_{0}^{\infty}f(t+x)x^{-1/\alpha}dx,\;\alpha>1.
$$
Let us study   a Cauchy problem
\begin{equation}\label{23a}
     \mathfrak{D}^{1/\alpha}_{-}  u = \varphi(W) u  ,\;u(0)=f\in \mathrm{D}(W^{n} ),\,\forall n\in \mathbb{N}.
\end{equation}
\begin{teo}\label{T4}
 Assume that conditions of  Theorem \ref{T3} holds, then
there exists a solution of the Cauchy problem \eqref{23a} in the form
\begin{equation}\label{25k}
u(t)= \sum\limits_{\nu=0}^{\infty} \mathcal{A}_{\nu}(\varphi^{\alpha},t)f.
\end{equation}
Moreover, the existing solution is unique if the operator   $\mathfrak{D}^{1-1/\alpha}_{-}\!\varphi(W)$ is accretive.
 \end{teo}
\begin{proof}
Let us show that $u(t)$ is a solution of the problem \eqref{23a},
   we need establish  the following relation
\begin{equation}\label{29}
\frac{d}{dt}\!\!\int\limits_{\vartheta(B)} \varphi(\lambda)^{ 1-\alpha }e^{-\varphi^{\alpha}(\lambda)t}B(I-\lambda B)^{-1}f\, d\lambda =
-\!\!\int\limits_{\vartheta(B)}\varphi(\lambda) e^{-\varphi(\lambda)^{\alpha}t}B(I-\lambda B)^{-1}f\, d\lambda,\,h\in \mathfrak{H}
\end{equation}
i.e. we can use a differentiation operation  under the integral. Using  simple  reasonings, we obtain the fact that
 that for an arbitrary
$$
\vartheta_{j}(B):=\left\{\lambda:\;|\lambda|=r>0,\, \theta_{0} \leq\mathrm{arg} \lambda \leq \theta_{1} \right\}\cup\left\{\lambda:\;r<|\lambda|<r_{j},\,\mathrm{arg} \lambda =\theta_{0}  ,\,\mathrm{arg} \lambda =\theta_{1} \right\},
$$
 there exists a limit
$
 (e^{-\varphi^{\alpha}(\lambda)  \Delta t} -1)e^{-\varphi^{\alpha}(\lambda)  t}/ \Delta t \,    \longrightarrow -\varphi^{\alpha}(\lambda) e^{-\varphi^{\alpha}(\lambda)  t}  ,\,\Delta t\rightarrow 0,
$
where convergence is uniform with respect to $ \lambda\in \vartheta_{j}(B).$     By virtue of the  decomposition on the Taylor series, we get
$$
\left |\frac{e^{-\varphi^{\alpha}(\lambda)  \Delta t} -1}{ \Delta t}e^{-\varphi^{\alpha}(\lambda)  t}\right | \leq |\varphi(\lambda)|^{\alpha} e^{ |\varphi(\lambda)  |^{\alpha}\Delta t}
e^{-\mathrm{Re}\,\varphi^{\alpha}(\lambda)  t}\leq |\varphi(\lambda)|^{\alpha} e^{(\Delta t-Ct)  |\varphi(\lambda)|^{\alpha}  }, \,\lambda\in \vartheta(B).
$$
Thus, applying the latter estimate,  Lemma 6 \cite{firstab_lit:1kukushkin2021}, for a sufficiently small value $\Delta t,$ we get
\begin{equation}\label{23m}
 \left\|\, \int\limits_{\vartheta(B)}\frac{e^{-\varphi^{\alpha}(\lambda)  \Delta t} -1}{ \Delta t}\varphi^{1-\alpha}(\lambda)e^{-\varphi^{\alpha}(\lambda)  t}B(I-\lambda B)^{-1}f d\lambda \right\|_{\mathfrak{H}}
  \leq
   C \|f\|_{\mathfrak{H}} \int\limits_{\vartheta(B)} e^{-C   t|\varphi(\lambda)|^{\alpha}}|\varphi(\lambda)|    |d\lambda|.
\end{equation}
Let us establish the convergence of the last integral. Applying Theorem \ref{T2a}, we get\\
$$
\int\limits_{\vartheta(B)} e^{-C   t|\varphi(\lambda)|^{\alpha}}|\varphi(\lambda)|    |d\lambda| \leq \int\limits_{\vartheta(B)}  e^{-t    e^{C|\lambda|^{\varrho(|\lambda|)}}  }    e^{ C|\lambda|^{\varrho }}    |d\lambda|.
$$
It is clear that the latter integral is convergent for an arbitrary positive value $t,$ what     guaranties that the improper integral at the left-hand side of \eqref{23m}
 is uniformly convergent with respect to $\Delta t.$ These facts give us an opportunity to claim that the relation \eqref{29} holds. Here, we should explain that this conclusion is based upon the generalization of the well-known theorem of the calculus. In its own turn it follows easily from the theorem on the connection with the simultaneous limit and the repeated limit. We left a complete  investigation of the matter to the reader having noted that the scheme of the reasonings is absolutely analogous in comparison with the ordinary calculus.

 Applying  the scheme of the proof corresponding to the  ordinary integral calculus,  using the contour $\vartheta_{j}(B),$    applying  Lemma  6  \cite{firstab_lit:1kukushkin2021} respectively, we can establish a  formula
\begin{equation}\label{18}
\int\limits_{0}^{\infty}x^{-\xi}dx\!\!\int\limits_{\vartheta(B)}e^{-\varphi^{\alpha}(\lambda)(t+x)} B(I-\lambda B)^{-1}f\, d\lambda=\!\!\int\limits_{\vartheta(B)}e^{-\varphi^{\alpha}(\lambda)t} B(I-\lambda B)^{-1}f d\lambda\int\limits_{0}^{\infty}x^{-\xi}e^{-\varphi^{\alpha}(\lambda)x}dx,
\end{equation}
where $\xi\in(0,1).$
 Applying  the obtained  formulas, taking into account a relation
$$
  \int\limits_{0}^{\infty}x^{-1/\alpha}e^{-\varphi^{\alpha}(\lambda)x}dx=  \Gamma(1-1/\alpha) \varphi^{1-\alpha}(\lambda),
$$
we get
\begin{equation}\label{19}
\mathfrak{D}^{1/\alpha}_{-}\!\!\! \int\limits_{\vartheta(B)}e^{-\varphi^{\alpha}(\lambda)t}     B(I-\lambda B)^{-1}f= \!\!\int\limits_{\vartheta(B)}e^{-\varphi^{\alpha}(\lambda)t}  \varphi(\lambda)  B(I-\lambda B)^{-1}f\, d\lambda.
\end{equation}
Applying Lemma \ref{L3a}, relation \eqref{19m}, we obtain the fact that $u$ is a solution of the equation \eqref{23a}.
 The fact that the initial condition holds,  in the   sense
$
u(t)   \xrightarrow[   ]{\mathfrak{H}}  f,\,t\rightarrow+0,
$
follows from the second relation \eqref{26} Theorem \ref{T3}.  The scheme of the proof corresponding to the uniqueness part is given in Theorem 6 \cite{firstab_lit:2kukushkin2022}.
  We complete the proof.
\end{proof}

\noindent{\bf Applications to  concrete operators and physical processes}\\

Note that  the made approach allows us to obtain a solution  for the evolution equation with the operator function in the second term where the operator-argument belongs to a sufficiently wide class of operators. In this regard, one can find a plenty of examples  in the paper  \cite{firstab_lit:2kukushkin2022} where such well-known operators as the Riesz potential,  the Riemann-Liouville fractional differential operator, the Kipriyanov operator,  the difference operator are considered, some interesting  examples that cannot be covered by the results presented in  \cite{firstab_lit:Shkalikov A.}  can be also  found in the paper  \cite{firstab_lit(arXiv non-self)kukushkin2018}.
 More general approach, implemented in the paper \cite{kukushkin2021a}, allows us to  produce an abstract example -- a transform of an operator belonging to the class of  m-accretive operators.   We should stress a significance of the last claim  since the  class    contains the   infinitesimal generator of a strongly continuous    semigroup of contractions. In its own turn  fractional differential operators of the real order can be expressed in terms of the  infinitesimal generator of the corresponding semigroup,  what makes the offered generalization relevant   \cite{kukushkin2021a}.
 Application of the obtained results  appeals to  electron-induced kinetics of ferroelectrics polarization
switching as the self-similar memory physical systems. The whole point is that  the mathematical
model of the fractal dynamic system includes an initial value problem for the fractional order
differential equation  considered in the paper \cite{L. Mor}, where    computational schemes for solving fractional differential problem
were constructed using Adams-Bashforth-Moulton type predictor-corrector methods.
The stochastic algorithm based on Monte-Carlo method was proposed to simulate the
domain nucleation process during restructuring domain structure in ferroelectrics. At the same time the results obtained in this paper allow us not only to solve the problem analytically but consider a whole class of   problems for evolution equations of fractional order. As for the mentioned concrete case \cite{L. Mor}, we just need   consider a suitable functional  Hilbert  space and apply Theorem \ref{T4} directly. For instance, it can be the  Lebesgue space of square-integrable functions.  Here, we should note that in the case corresponding to a functional Hilbert space we gain more freedom in constructing the theory, thus some modifications of the method can appear but it is an issue for further more detailed study what is not supposed  in the framework of this paper. However, the following concrete case may be of interest to the reader.

  Goldstein et al. proved in \cite{firstab_lit:Goldstein} several new result of this type replacing the
Laplacian by the Kolmogorov operator
$$
L=\Delta+\frac{\nabla\rho}{\rho}\cdot\nabla.
$$
Here $\rho$ is a probability density on $\mathbb{R}^{N}$ satisfying $\rho\in C^{1+\alpha}_{lok}(\mathbb{R}^{N})$
  for some
$\alpha\in (0, 1),\; \rho(x) > 0$ for all $x\in \mathbb{R}^{N}.$ A reasonable question can appear  - Is there possible connections between
the developed theory   and the operator $L$? Indeed, the mentioned operator gives us an opportunity to show briskly capacity of the spectral theory methods. First of all, let us note the following relation holds
$$
L=\frac{1}{\rho}W,
$$
where
$$
 Wf:=D^{j} \rho D_{j} f = \mathrm{div} \rho  \nabla f,
$$
thus the right direction of the issue investigation should be connected with the operator composition $ \rho^{-1} W$ since the operator $W$ is uniformly elliptic and satisfies hypophyses $\mathrm{H}1,\mathrm{H}2$ \cite{kukushkin2021a} thus the results \cite{firstab_lit(arXiv non-self)kukushkin2018}, \cite{kukushkin2021a}, \cite{firstab_lit:1kukushkin2021} can be applied to the operator. A couple of words on the difficulties appearing while we study the operator composition. From the first glance it looks pretty well, but it is not so for the inverse operator (one need prove that it is a resolvent)  is a composition of an unbounded operator and a resolvent of the operator $W,$  indeed  since $R_{W}W=I,$ then formally, we have
$$
L^{-1}f= R_{W}\rho f.
$$
At the same time I think that the general theory created in the papers can be adopted to some operator composition but it is  a tremendous work. Instead of that I suggest that we should find a suitable pair of Hilbert spaces that is also not so easy matter))). However, we will see! Bellow, we consider a space $\mathbb{R}^{N}$ endowed with the norm
 $$
 |x|=\sqrt{\sum\limits_{k=1}^{n}|x_{k}|^{2}},\,x=(x_{1},x_{2},...,x_{n} )\in \mathbb{R} ^{N}.
 $$
 Assume that there exists a constant $\lambda>2$ such that  the following condition holds
$$
\left\| \rho^{1/\lambda-1}\nabla\rho \right\|_{L_{\infty}(\mathbb{R}^{N})}<\infty,\;\rho^{1/\lambda}(x)=O(1+|x|).
$$
One can verify easily that this condition is not unnatural for it holds for a function $\rho(x)=(1+|x|)^{\lambda},\,x\in \mathbb{R}^{N},\,\lambda\geq 1.$ Let us define a space $\mathfrak{H}_{+}$ as a completion of the  set $C_{0}^{\infty}(\mathbb{R}^{N})$ with the norm
$$
\|f\|^{2}_{\mathfrak{H}_{+}}=\|\nabla f\|^{2}_{L_{2}(\mathbb{R}^{N})}+\|f\|^{2}_{L_{2}(\mathbb{R}^{N},\varphi^{-2})},\,\varphi(x)=(1+|x|),
$$
here one can easily see that it is generated by the corresponding inner product in the Hilbert  space $\mathfrak{H}_{+}.$ The following result can be obtained as a consequence of the Adams theorem (see Theorem 1 \cite{firstab_lit:1Adams}).
Note that
$$
 \varphi^{\lambda/2}\nabla f=\nabla(f\varphi^{\lambda/2})-f\nabla \varphi^{\lambda/2},\,f\in C_{0}^{\infty}(\mathbb{R}^{N}),
$$
therefore in accordance with the triangle inequality, we have
$$
\left(\;\int\limits_{\mathbb{R}^{N}}|\varphi^{\lambda/2}\nabla f|^{2}dx\right)^{1/2}\leq \left(\;\int\limits_{\mathbb{R}^{N}}|\nabla(f\varphi^{\lambda/2})|^{2}dx\right)^{1/2}
 +\left(\;\int\limits_{\mathbb{R}^{N}}|f\nabla \varphi^{\lambda/2}|^{2}dx\right)^{1/2}=
$$
$$
=\left(\;\int\limits_{\mathbb{R}^{N}}|\nabla(f\varphi^{\lambda/2})|^{2}dx\right)^{1/2}
 +\frac{\lambda}{2}\left(\;\int\limits_{\mathbb{R}^{N}}|f\varphi^{\lambda/2}|^{2}\varphi^{-2}dx\right)^{1/2},
$$
here we have  substituted the relation
$
\nabla\varphi^{\lambda/2}= \varphi^{\lambda-2}\lambda^{2}/4
$
that can be obtained by direct calculation. Therefore, we get
$$
\left(\;\int\limits_{\mathbb{R}^{N}}|\nabla f|^{2}\varphi^{\lambda }dx\right)^{1/2}\leq C\left(\;\int\limits_{\mathbb{R}^{N}}|\nabla(f\varphi^{\lambda/2})|^{2}dx+\int\limits_{\mathbb{R}^{N}}|f\varphi^{\lambda/2}|^{2}\varphi^{-2}dx\right)^{1/2}.
$$
The letter relation can be rewritten in terms of the norm
$$
\left(\;\int\limits_{\mathbb{R}^{N}}|\nabla (g\varphi^{-\lambda/2 })|^{2}\varphi^{\lambda }dx\right)^{1/2}\leq C\|g\|_{\mathfrak{H}_{+}},\,g=f \varphi^{ \lambda/2 },
$$
Therefore, in accordance with  the Theorem 1 \cite{firstab_lit:1Adams}, we conclude that if a  set is bounded in the sense of the norm  $\mathfrak{H}_{+}$ then it  is compact in the sense of the norm
$$
\left(\;\int\limits_{\mathbb{R}^{N}}  |g|^{2} \varphi^{-\lambda  }  dx\right)^{1/2}.
$$
It is clear that
$$
 \int\limits_{\mathbb{R}^{N}}  |g|^{2} \varphi^{-\lambda  }  dx\leq \int\limits_{\mathbb{R}^{N}}  |g|^{2} \varphi^{-2 }  dx,\,\lambda>2.
$$
Therefore, we have the following inequality
$
\|g\|_{\mathfrak{H}_{-}}\leq   \|g\|_{\mathfrak{H}_{+}},\,g\in \mathfrak{H}_{+},
$
where
$$
\|g\|_{\mathfrak{H}_{-}}:=\left(\;\int\limits_{\mathbb{R}^{N}}  |g|^{2} \varphi^{-\lambda  }  dx\right)^{1/2} .
$$
Thus, we have created a pair of Hilbert spaces $\mathfrak{H}_{-}$ and $\mathfrak{H}_{+}$ satisfying the condition of compact embedding. Let us see how can it help us in studying the operator $L.$ Consider an operator $L':=-L+ \rho^{-2/\lambda}I,$   we ought to remark here that we need involve additional summand to apply the methods \cite{kukushkin2021a}. The crucial point is related to how to estimate the second term of the operator $-L$ from bellow. Here, we should point out that some peculiar techniques of the theory of functions can be involved. However, along with this we can consider simplified case (since we have imposed additional conditions upon the function $\rho$) in order to show how the invented method works. The following reasonings are made under the assumption  that the functions  $f,g\in C_{0}^{\infty}(\mathbb{R}^{N}).$ Consider
$$
\left|\,\int\limits_{\mathbb{R}^{N}}\frac{\nabla\rho}{\rho}\cdot\nabla f\,\bar{g}dx\right|\leq \,\int\limits_{\mathbb{R}^{N}}\left|\frac{\nabla\rho}{\rho}\right| |\nabla f|\,|g|dx \leq  \left\| \rho^{1/\lambda-1}\nabla\rho \right\|_{L_{\infty}(\mathbb{R}^{N})}\int\limits_{\mathbb{R}^{N}}  |\nabla f|\,|g| \rho^{-1/\lambda}dx\leq
$$
$$
 \leq    \left\| \rho^{1/\lambda-1}\nabla\rho \right\|_{L_{\infty}(\mathbb{R}^{N})}
   \|\nabla f\|_{L_{2}(\mathbb{R}^{N})} \|g\|_{L_{2}(\mathbb{R}^{N},\rho^{-2/\lambda})}\leq
$$
$$
\leq    \left\| \rho^{1/\lambda-1}\nabla\rho \right\|_{L_{\infty}(\mathbb{R}^{N})}\frac{1}{2}\left\{\|\nabla f\|^{2}_{L_{2}(\mathbb{R}^{N})}+ \|g\|^{2}_{L_{2}(\mathbb{R}^{N},\rho^{-2/\lambda})}  \right\}.
$$
Therefore
$$
-\mathrm{Re}\left(\frac{\nabla\rho}{\rho}\cdot\nabla f ,f \right)_{L_{2}(\mathbb{R}^{N})}\geq -\left\| \rho^{1/\lambda-1}\nabla\rho \right\|_{L_{\infty}(\mathbb{R}^{N})}\frac{1}{2}\left\{\|\nabla f\|^{2}_{L_{2}(\mathbb{R}^{N})}+ \|f\|^{2}_{L_{2}(\mathbb{R}^{N},\rho^{-2/\lambda})}  \right\}
$$
and in the case $ \left\| \rho^{1/\lambda-1}\nabla\rho \right\|_{L_{\infty}(\mathbb{R}^{N})}<2,$ we get easily
$$
\mathrm{Re}\left(L' f ,f \right)_{L_{2}(\mathbb{R}^{N})}\geq C_{0}\|f\|^{2}_{\mathfrak{H}_{+}},\,C_{0}>0.
$$
Using the above estimates, we obtain
$$
\left|(L' f ,g  )_{L_{2}(\mathbb{R}^{N})}\right|\leq  \|\nabla f\| _{L_{2}(\mathbb{R}^{N})}\|\nabla g\| _{L_{2}(\mathbb{R}^{N})}+\left\| \rho^{1/\lambda-1}\nabla\rho \right\|_{L_{\infty}(\mathbb{R}^{N})}
   \|\nabla f\|_{L_{2}(\mathbb{R}^{N})} \|g\|_{L_{2}(\mathbb{R}^{N},\rho^{-2/\lambda})}
 $$
 $$
  +\|f\| _{L_{2}(\mathbb{R}^{N},\rho^{-2/\lambda})}\|g\| _{L_{2}(\mathbb{R}^{N},\rho^{-2/\lambda})}\leq C_{1}\|f\|_{\mathfrak{H}_{+}}\|g\|_{\mathfrak{H}_{+}},\,C_{1}>0.
$$
Thus, we have a fulfilment of the   hypothesis $\mathrm{H}2$ \cite{kukushkin2021a}. Taking into account a fact  that a  negative space  $L_{2}(\mathbb{R}^{N},\varphi^{-\lambda})$ is involved,
we are being compelled to involve a modification of the hypothesis  $\mathrm{H}1$ \cite{kukushkin2021a} expressed  as follows. There  exist pairs of Hilbert spaces   $ \mathfrak{H}_{-}\subset \mathfrak{H},\,\,\mathfrak{H}_{+}\subset \subset \mathfrak{H}_{-},$  where the latter symbol denotes a compact embedding of spaces. However, we can go further and  modify a norm $\mathfrak{H}_{+}$ adding a summand in this case the considered operator can be changed, we have
$$
\|f\|^{2}_{\mathfrak{H}_{+}}:=\|\nabla f\|^{2}_{L_{2}(\mathbb{R}^{N})}+\|f\|^{2}_{L_{2}(\mathbb{R}^{N},\psi)},\,\psi(x)=(1+|x|)^{-2}+1,\,L':=L+I.
$$
Implementing the same reasonings one can prove that in this case the hypophyses $\mathrm{H}2$ \cite{kukushkin2021a} is fulfilled, the modified analog of the hypothesis  $\mathrm{H}1$ \cite{kukushkin2021a} can be formulated as follows.\\

 \noindent ($ \mathrm{H}1^{\ast} $) There  exists a chain  of Hilbert spaces   $\mathfrak{H}_{+}\subset  \mathfrak{H} \subset \mathfrak{H}_{-},\,\mathfrak{H}_{-}\subset\subset \mathfrak{H} $ and a linear manifold $\mathfrak{M}$ that is  dense in  $\mathfrak{H}_{+}.$ The operator $L$ is defined on $\mathfrak{M}.$\\

However, we have $\mathfrak{H}_{+}\subset\subset \mathfrak{H}_{-}$ instead of required $\mathfrak{H}_{+}\subset\subset \mathfrak{H}.$  This inconvenience can stress a peculiarity of the chosen method, at the same time the central point of the theory - Theorem 1 \cite{kukushkin2021a}  can be reformulated under newly obtained conditions corresponding to both variants of the operator $L'.$      The further step is how to calculate order of the operator $L'.$   Here, we should point out that there exists the Fefferman concept that covers such a kind of problems. For instance, the Rozenblyum result is presented in the monograph
\cite[p.47]{firstab_lit:Rosenblum}, in accordance with which we can choose such an unbounded subset of $\mathbb{R}^{N}$ that the relation $\lambda_{j}(\mathrm{Re} L')\asymp j^{2/N}$ holds, where the symbol $\lambda_{j}$ denotes an eigenvalue.   Thus, we left this question to the reader for  a more detailed  study    and reasonably allow ourselves to  assume that the  condition $\mu(\mathrm{Re} L')=2/N$ holds, where the symbol $\mu$ denotes operator order (see \cite{kukushkin2021a}). Having obtained analog of Theorem 1 \cite{kukushkin2021a} and order  of the real component of the operator $L'$ we have a key to the theory created in the papers \cite{firstab_lit:1kukushkin2021},\cite{firstab_lit:2kukushkin2022},\cite{firstab_lit(axi2022)}. Thus we can consider a Cauchy problem for the evolution equation with the operator $L'$ in the second term as well as a function of the operator $L'$ in the second term what leads us to the integro-differential evolution equation - it corresponds to an operator function having   finite  principal and major parts of the corresponding Laurent series.\\

\section{Conclusions}

In this paper, we invented a method to study Cauchy problems for abstract fractional  evolution equations with the operator function in the  second term. The considered class corresponding to the operator-argument is rather wide and includes non-selfadjoint unbounded operators.
However, as  a main result we  represent  an approach  allowing us to principally broaden conditions imposed upon the second term of the  abstract fractional evolution equation. The application part of the paper appeals to the theory of fractional differential equations.  In particular,    the existence and uniqueness theorems  for fractional evolution  equations,  with the second term  being presented by  an operator function of a differential operator with a fractional derivative in  final terms, are covered by the invented abstract method. In connection with this, various types of  fractional integro-differential operators  can be considered, it becomes clear if we involve an operator function represented by the Laurent series with finite principal and regular parts.
       We hope that the general concept will have a more detailed study as well as  concrete applied problems will be solved by virtue of the invented theoretical approach.\\

\noindent{\bf Acknowledgment  }\\

I am sincerely grateful to my scientific colleagues for discussions on related problems, I express my gratitude to B.G. Vakulov,  Yu.E. Drobotov, T.M. Andreeva, Yildirim Ozdemir.

\end{document}